\documentclass{amsart}
\usepackage{amsmath,amsthm,amssymb}
\usepackage{graphicx}
\usepackage{tikz-cd}
\usepackage{paralist}
\usepackage{environ}
\usepackage{xcolor}
\usepackage{hyperref}
\usepackage{centernot}
\usepackage{mathtools}
\usepackage{marvosym}
\usepackage{xcolor,cancel}

\usepackage{verbatim}
\usepackage{csquotes}
\usepackage{pstricks}
\usepackage{mdframed}

\makeatletter
\def\square{\pst@object{square}}% reads star and options and continues with \square@i
\def\square@i(#1,#2)#3{{\use@par\solid@star\psframe[origin={#1,#2}](#3,#3)}}
\makeatother

\makeatletter
\DeclareFontFamily{U}{tipa}{}
\DeclareFontShape{U}{tipa}{bx}{n}{<->tipabx10}{}
\newcommand{\arc@char}{{\usefont{U}{tipa}{bx}{n}\symbol{62}}}%

\newcommand{\arc}[1]{\mathpalette\arc@arc{#1}}

\newcommand{\arc@arc}[2]{%
	\sbox0{$\m@th#1#2$}%
	\vbox{
		\hbox{\resizebox{\wd0}{\height}{\arc@char}}
		\nointerlineskip
		\box0
	}%
}
\makeatother

\makeatletter
\newcommand{\doublewedge}{\big@doubleop{\wedge}}
\newcommand{\big@doubleop}[1]{%
	\DOTSB\mathop{\mathpalette\big@doubleop@aux{#1}}\slimits@
}

\newcommand\big@doubleop@aux[2]{%
	\sbox\z@{$\m@th#1#2$}%
	\makebox[1.35\wd\z@][s]{$\m@th#1#2\hss#2$}%
}

\newcommand{\dnear}{\delta_{\Phi}} % descriptive proximity
\usepackage{subfigure}
\renewcommand{\thesubfigure}{\thefigure.\arabic{subfigure}}
\makeatletter
\renewcommand{\p@subfigure}{}
\renewcommand{\@thesubfigure}{\thesubfigure:\hskip\subfiglabelskip}
\makeatother

\theoremstyle{plain}
\newtheorem{theorem}{Theorem}
\newtheorem{lemma}{Lemma}

\newtheorem{definition}{Definition}

\newtheorem{example}{Example}
\newtheorem{corollary}{Corollary}
\newtheorem{proposition}{Proposition}

\usepackage{cancel}
\newcommand{\near}{\delta} 
\usepackage[title]{appendix}

\begin{document}
	\title{Approximately Prime Rings and Prime Ideals}
	\author{Maram Almahariq}
	\address{Department of Mathematics, Birzeit university, Ramallah, Palestine,
	}
	\email{maram.mahareeq.14@gmail.com}
	
	\author[J.F. Peters]{J.F. Peters}
	\address{
		Department of Electrical and Computer Engineering,
		University of Manitoba, WPG, Manitoba, R3T 5V6, Canada and
		Department of Mathematics, Faculty of Arts and Sciences, Ad\.{i}yaman University, 02040 Ad\.{i}yaman, Turkey,
	}
	\email{james.peters3@umanitoba.ca}
	\thanks{The research has been supported by the Natural Sciences \&
		Engineering Research Council of Canada (NSERC) discovery grant 185986 
		and Instituto Nazionale di Alta Matematica (INdAM) Francesco Severi, Gruppo Nazionale per le Strutture Algebriche, Geometriche e Loro Applicazioni grant 9 920160 000362, n.prot U 2016/000036 and Scientific and Technological Research Council of Turkey (T\"{U}B\.{I}TAK) Scientific Human
		Resources Development (BIDEB) under grant no: 2221-1059B211301223.}
	
	\author[T. Vergili]{T. Vergili}
	\address{
		Department of Mathematics, Karadeniz Technical University, Trabzon, Turkey,
	}
	\email{tane.vergili@ktu.edu.tr}
	\date{}

\begin{abstract}
This article focuses on approximately prime rings and approximately prime ideals in proximal relator spaces, especially in descriptive proximity spaces. In particular, we define some binary operations, including the product of two approximately prime ideals and the direct product of approximately prime rings, and study the approximately principle prime ideals. Moreover, we introduce some fundamental properties of these approximately algebraic structures.
\end{abstract}
	\maketitle
	\tableofcontents

{
\section{Introduction}
The concept of prime rings and prime ideals has played an important role in the theory of abstract algebra, mainly in commutative rings. Briefly, a {\bf commutative ring}~\cite[\S III.3, p. 66]{Birkhoff1965} closed under two binary operations, commutative and associative operations such that
\begin{compactenum}[1$^o$]
\item Multiplication is distributive over addition.
\item An additive identity and additive inverse exist.
\end{compactenum}
In this paper we aim to obtain algebraic structures in proximal relator spaces which are, for $X$ is a non-empty set and $\mathcal{R}$ is a family of proximity relations on $X$, then $(X,\mathcal{R}_{\delta})$ is called a proximal relator space where $R_{\delta} $ contains proximity relations \cite{Peters2016relator}, using the definition of approximately rings and approximately ideals defined by Inan \cite{Inan:2017}. \newline
Essentially, the motivation of this article is to define the proximately principle prime ideals, the product of two approximately prime ideals, the approximately direct product of prime rings, and we investigate consequences of the results for these structures.  

}
	
\section{Background}
 This section introduces 
 For a non-empty set $X$ and subsets $A, B\subseteq X$, we say that $A$ is near $B$, provided $A,B$ have one or more points in common or they contain one or more points that are sufficiently close to each other \cite{Peters2016ComputationalProximity} (see, also,~\cite{PetersGuadagni2016}).   A non-empty set $X$ together with the relation $\delta$ is a \v{C}ech proximity space, denoted by $(X,\delta)$, provided the following axioms are satisfied \cite{Cech1966}.

\begin{description}
  \item[({\bf P}.0)]  $A \not{\near}\ \emptyset $, for all $A \subseteq X$. 
  \item[({\bf P}.1)] $A \near B$ if and only if $B \near A$
   \item[({\bf P}.2)] $A \cap B \neq \emptyset$ implies $A \near B$.
   \item[({\bf P}.3)] $A \near (B\cup C) $ if and only if $A \near B$ or $ A \near C$.
\end{description}

Efremovi\c{c} \cite{Efremovic1952} added an axiom called the (Efremovi\v{c} Axiom) to the \v{C}ech axioms, namely, 
 \[	
 A \not{\near}\ B, \ \exists C \subseteq X \ \text{such that} \ A \not{\near}\ C \ \text{and} \ C^{c} \not{\near}\ B.
 \]
	
An incisive introduction to Efremovi\c{c} proximities is given by C. Guadagni~\cite{Guadagni2015}. For an introduction to the geometry and topology of strong proximities, see~\cite{PetersGuadagni2016}.
Lodato's \cite{Lodato1962} proximity shifts the Efremovic axiom to 
	\[
\text{If}	 \ A \ \delta \ B\ \text{and}\ \forall b\in B, \; \{b\} \ \delta \ C \; \text{then} \; A \ \delta \ C.
	\]
	
Let $\phi:X \to \mathbb{R}$ be a map which represents the properties of a point in set $X$ such as density, color or texture \cite{Peters2013mcsintro,Peters2016ComputationalProximity}. A probe map $\Phi:X \to \mathbb{R}^n$ is a collection of maps $(\phi_1, \phi_2, \ldots, \phi_n)$ with $\phi_i: X \to \mathbb{R}$ for $i \in \{1,2,\ldots,n\}$. For $p \in X$, the vector  $\Phi(p)=(\phi_1 (p), \phi_2 (p),\cdots, \phi_n (p))$ is called a feature vector of $p$. For $A \subseteq X$, $\Phi(A)=(\phi_1 (A), \phi_2 (A), …, \phi_n (A))$ is called a feature vector of $A$. Let $p, p^\prime  \in X$. We say that $\{p\}$ is descriptively near to $\{p^\prime \}$, denoted by $\{p\}  \ \delta_\Phi \ p^\prime$, iff $\Phi(p)= \Phi(p^\prime)$ and  \cite{Peters2016ComputationalProximity}.
	
	\[ A\mathrel{\delta_\Phi}B \Leftrightarrow \Phi(A)\cap\Phi(B)\neq\emptyset.\] 
	Let $(X,\delta_{\Phi})$ be a descriptive proximity space. The descriptive intersection of non-empty subsets $A$ and $B$, denoted by $A \mathop{\cap}\limits_{\Phi} B \neq \emptyset$, is defined as \cite{DiConcilio2018MCSdescriptiveProximities}
	
	\[A \mathop{\cap}\limits_{\Phi} B = \{x\in A \cup B :  \Phi(x) \in \Phi(A) \cap \Phi(B)\}.\]
	
\begin{definition}  {\bf (Descriptive Proximity Space)}
\cite{DiConcilio2018MCSdescriptiveProximities} \\
	Let $X$ be a non-empty set equipped with the relation $\dnear$. $X$ is a descriptive proximity space if and only if the following descriptive forms of the \v{C}ech axioms are satisfied, for all non-empty subsets $A$, $B$, and $C$ of $X$.		
\begin{description}
	\item[({\bf DP}.0)]  $A \not\dnear \ \emptyset$, for all $A \subseteq X$. 
	\item[({\bf DP}.1)] $A \ \dnear \ B $ if and only if $B \ \dnear \ A$
	\item[({\bf DP}.2)] $A  \mathop{\cap}\limits_{\Phi} B \neq \emptyset$ if and only if $A \  \dnear \ B$.
	\item[({\bf DP}.3)] $A \ \dnear (B\cup C) $ if and only if $A  \ \dnear \ B$ or $A \ \dnear \ C$. \qquad \textcolor{blue}{$\blacksquare$}
\end{description} 
\end{definition} 
	
\begin{definition} \cite{Peters2016relator} 
	Let $X$ be a non-empty set and $\mathcal{R}$ be a family of proximity relations on $X$, then $(X,\mathcal{R}_{\delta})$ is called a proximal relator space where $R_{\delta} $ contains proximity relations, for example \v{C}ech proximity, Efremovi\v{c} proximity, Lodato proximity, or descriptive proximity.    \qquad \textcolor{blue}{$\blacksquare$}
\end{definition}
	
In this paper, we concentrate on descriptive proximity relations.
\begin{definition} {\bf (Descriptively upper approximation of a set)} \cite{Inan:2017} \\ 
	Let $(X,\mathcal{R}_{\dnear})$ be a descriptive proximal relator space and $A \subseteq X$. A descriptively upper approximation of $A$ is defined by
	\[		
	\Phi^{*}A = \{ x \in X \ \ : \ \  \{x\} \ \dnear \ A  \}.
	\]  \qquad \textcolor{blue}{$\blacksquare$}
\end{definition}
	
\begin{definition} {\bf (Approximately groupoid)} \cite{Inan:2017} \\
  For a descriptive proximal relator space $(X,R_{\delta_\phi})$. Let \enquote{$\cdot$}  be a binary operation on $X$ and $G \subseteq X$ is called an approximately groupoid, provided for all $x,y \in G$,  $x \cdot y \in \Phi^{*}G$.  \qquad \textcolor{blue}{$\blacksquare$}
\end{definition}
	
\begin{definition} {\bf (Approximately group)}
  \cite{Inan:2019} \\  Given a descriptive proximal relator space $(X,R_{\delta_\phi})$ let  \enquote{$\cdot$} be a binary operation on $X$. A subset $G \subseteq X$ is called an approximately group, if the following conditions are satisfied. \\
  \begin{description}
	\item[(${\bf AG}_1$)] \boxed{x\cdot y \in \Phi^{*}G} for all $x,y \in G$
	\item[(${\bf AG}_2$)] The property \boxed{(x\cdot y)\cdot z = x\cdot(y\cdot z)} holds in $\Phi^{*}G$ for all $x,y,z \in G$.
	\item[(${\bf AG}_3$)] There exists $e \in \Phi^{*}G$ such that \boxed{x\cdot e=e\cdot x=x} for all $x \in G$. 	
	\item[(${\bf AG}_4$)] For all $x\in G$, there exists $y \in G$ such that \boxed{x\cdot y= y\cdot x =e.}
\end{description}
		\qquad \textcolor{blue}{$\blacksquare$}
\end{definition}
	
Here, we say that $G$ is approximately semigroup \cite{Inan:2017},  if it satisfies only (${\bf AG}_1$) and (${\bf AG}_2$).  	

\begin{definition} {\bf (Approximately ring)} \cite{Inan2019TJMdescriptiveProximity} \\ 
Let $(X,\mathcal{R}_{\dnear})$ be a descriptive proximal relator space and let \enquote{+} and \enquote{$\cdot$} be binary operations on $X$. $R \subseteq X$ is called an approximately ring, provided  the following conditions are satisfied.
\begin{enumerate}
	\item[$({\bf AR}_1)$] $R$ is an abelian approximately group with \enquote{+}.
	\item[$({\bf AR}_2)$] $R$ is an  approximately semigroup with \enquote{$\cdot$}.
	\item[$({\bf AR}_3)$] For all $x,y,z \in R$, \boxed{x\cdot(y+z) =(x\cdot y)+(x\cdot z)$ and $(x+y)\cdot z = (x\cdot z)+(y\cdot z)}	hold in $\Phi^{*}R$.
\end{enumerate}
In addition,
\begin{enumerate}
  \item[$({\bf AR}_4)$] If \boxed{x\cdot y = y\cdot x} for all $x,y \in R$, then $R$ is called a commutative approximately ring. 
  \item[$({\bf AR}_5)$] If $\Phi^{*}R$ contains an element $1_R$ such that \boxed{1_R \cdot x = x \cdot 1_R = x} for all $x \in R$, then $R$ is called an approximately ring with unity (identity). \qquad \textcolor{blue}{$\blacksquare$}
\end{enumerate} 
\end{definition}

\begin{definition} {\bf (Invertible element)} \cite{Inan2019TJMdescriptiveProximity} \\ 
An element $x$ in an approximately ring $R$ is called left (resp. right) invertible, provided there exists $y \in R$ (resp, $z \in R)$ such that $y \cdot x = 1_{R}$ (resp, $x \cdot z = 1_R$. The element $y$ (resp. $z$) is called a left (resp. right) inverse of $x$. If $x \in R$ is both left and right invertible, then $x$ is called approximately invertible or approximately unit.    \qquad \textcolor{blue}{$\blacksquare$}
\end{definition}

\begin{definition} {\bf (Approximately Subring) \cite{Inan2019TJMdescriptiveProximity}} \\
	Let $(X,R_{\delta_\phi})$ be a descriptive proximal relator space and $R \subseteq X$ be an approximately ring. A non-empty subset $S$ of R is called an approximately subring, provided $S$ is an approximately ring with binary operations 	\enquote{+} and \enquote{$\cdot$} on approximately ring $R$.  \qquad \textcolor{blue}{$\blacksquare$}
\end{definition}

\begin{definition} {\bf (Approximately Field)}
  \cite{Inan2019TJMdescriptiveProximity} \\ An approximately ring $R$ is an approximately field, provided $({R}\setminus\{0\}, \cdot)$ is a commutative approximately group.  \qquad \textcolor{blue}{$\blacksquare$}
\end{definition}
	
\begin{definition} {\bf (Approximately Subfield)} \cite{Inan2019TJMdescriptiveProximity} \\
  Let $R$ be an approximately field and $S \subseteq R$. $S$ is called an approximately subfield of $R$, if $S$ is an approximately field.     \qquad \textcolor{blue}{$\blacksquare$}
\end{definition}

\begin{definition}  {\bf (Approximately Ideal)} \cite{Inan2019TJMdescriptiveProximity} \\ 
	Let $(X,\mathcal{R}_{\delta_\Phi})$ be a descriptive relator space, $R \subseteq X$ be an approximately ring and $I \subseteq R$. $I$ is a left (right) approximately ideal of $R$ if and only if for all $x , y \in I$ and for all $r \in R$, $x + y \in \Phi^{*} I$,$-x \in I$ and $r \cdot x \in \Phi^{*} I$ $(x \cdot r \in \Phi^{*}I)$. If $I$ is both left and right ideal, then $I$ is called an approximately ideal. \qquad \textcolor{blue}{$\blacksquare$}
\end{definition}

\begin{definition} \cite{Inan2019TJMdescriptiveProximity} Consider the set 
  \[R/_\rho S = \{ x + S \ \ : \ \  x \in R \}\]
  where $\rho$ is an equivalence relation. Then, $R/_\rho S$ is called an approximately ring of all approximately cosets of $R$ by $S$. \newline
  In addition, the descriptively upper approximation of $R/_\rho S$ is defined by,
   \[(\Phi^{*} R)/_\rho S = \{ x + S | x \in \Phi^{*} R \}.\]
\end{definition}

\section{Main results}
	
\begin{definition}  {\bf (Approximately Prime Ideal)} \\
 The approximately ideal $P$ in approximately ring $R$ is approximately prime ideal, provided for $a, b \in$ $R$, $ab \in \Phi^{*}P$, implies $a \in P$ or $b \in P$. 
\end{definition} 
	 
\begin{example}
 Let $I$ represent a digital image with $\delta_\Phi$, which consists of 16 pixels as illustrated in Figure~\ref{fig} in \cite{Inan2019TJMdescriptiveProximity}. Each pixel, labeled $x_{ij}$, is positioned at the coordinates $(i,j)$. The function $\Phi$ is a probe map that assigns each pixel to its corresponding $RGB$ (Red,Green, Blue) values, as indicated in the table.
\end{example}

\begin{figure}
	\centering
	\begin{subfigure}
		\centering
		\includegraphics[width=0.3\linewidth]{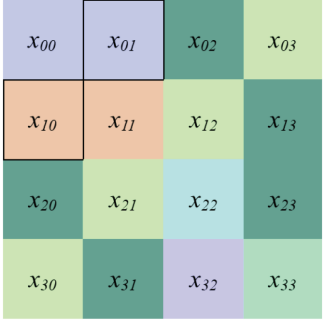}
	\end{subfigure}
	\begin{subfigure}
		\centering
		\includegraphics[width=0.45\linewidth]{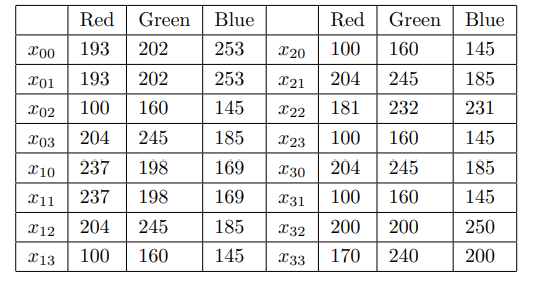}
	\end{subfigure}
	\caption{Digital Image $I$ and subimage $R_1$ together with RGB-value table \cite{Inan2019TJMdescriptiveProximity}.}
	\label{fig}
\end{figure}

Let $R_1 =\{x_{01},x_{10}\}$ and $\Phi^{*}R_1=\{x_{00},x_{01},x_{10},x_{11}\}$, then $R_1$ is an approximately ring under the following binary operations.
\begin{description}
	\item[Addition] The addition operation + isdefined as \[x_{ij} + x_{kl} = x_{nm},\]
	for $x_{ij}, x_{kl} \in R_1$, where $n \equiv i+k $ (mod 2), and $m \equiv j+l$ (mod 2).
	\item[Multiplication] 	The multiplication operation $*$ defined as
	\[x_{ij} * x_{kl} = x_{nm},\] for $x_{ij}, x_{kl} \in R_1$, where $n \equiv $ $min(i,k)$, and $m \equiv$ $min(j,l)$.
\end{description}
	Assume $I=\{x_{01}\}$ and $\Phi^{*} I =\{x_{00},x_{01}\}$, then $I$ is an approximately prime ideal. 
	However, if $R_2=\{x_{00},x_{01},x_{10},x_{11}\}$, then $\Phi^{*}R_2 = R_2$. Now consider $I = \{x_{01},x_{10}\}$ with $\Phi^{*} I = \Phi^{*} R_2$, notice that $x_{00}*x_{11} = x_{00} \in \Phi^{*}I$, but neither $x_{00}$ nor $x_{11}$ is element of $I$. Thus, $I$ is a non-approximately prime ideal.
	
\begin{definition}  {\bf (Approximately Integral Domain)} \\
	The non-zero approximately commutative ring $R$  is called an approximately integral domain if and only if the product of non-zero elements is non-zero (i.e., for $a$, $b$ $\in$ $R$
	and $ab$ $\in$ $\Phi^{*}$$R$ such that $ab=0$, then either $a=0$ or $b=0$). \qquad \textcolor{blue}{$\blacksquare$}
\end{definition}
	
\vspace*{0.3mm}
	
\begin{theorem}
	Let $I$ be an approximately ideal in an approximately ring $R$. If $I$ is an approximately prime ideal, then $R/I$ is an approximately integral domain. 
\end{theorem}
\begin{proof}
	Assume $I$ is an approximately prime ideal in an approximately ring $R$. We need to show that $R/I$ is an approximately integral domain. If $r_{1} + I$ and $r_{2} + I$ are elements in $R/I$ such that $(r_{1} + I)\cdot(r_{2} + I) =  0 + I = I$. We have $(r_{1} + I)\cdot(r_{2} + I) = r_{1}\cdot r_{2} + I =  I$, which implies $r_{1} \cdot r_{2} \in \Phi^{*} I$, since $I$ is an approximately prime ideal, so either $r_{1} \in I$, implies $r_{1} + I = I$ or $r_{2} \in I$, which implies $r_{2} + I = I$. Therefore, $R/I$ is an approximately integral domain.   
\end{proof}
	
\begin{definition} {\bf (Approximately Multiplicatively Closed Set)}
  Let $S$ be a subset of an approximately ring $R$, such that $S$ is called an approximately multiplicatively closed set, if and only if $S$ is not empty, $S$ does not contain the zero element of $R$, and if $a, b \in S$, then $a\cdot b \in \Phi^{*} S$. \qquad \textcolor{blue}{$\blacksquare$} 
\end{definition}
		
\begin{theorem} 
  Let $I$ be an approximately ideal in an approximately ring $R$, and $S = R\setminus\Phi^{*} I$ is a non-empty set. If $I$ is an approximately prime ideal, then $S$ is an approximately multiplicatively closed set.
\end{theorem}
\begin{proof}
   Suppose $I$ is an approximately prime ideal in an approximately ring $R$. The zero element belongs to $\Phi^{*} I$, hence $0 \notin S$. Now let $a,b \in S$, then $a,b \notin \Phi^{*} I$. So, either $a \cdot b \in \Phi^{*} S$ or $a\cdot b \notin \Phi^{*} S$. Assume $a\cdot b \notin \Phi^{*} S$, then $a\cdot b \in \Phi^{*} I$ and since $I$ is an approximately prime ideal. Thus, either $a \in I$ or $b \in I$. But this contradicts the fact that $a,b \notin \Phi^{*} I$. Therefore, $a\cdot b \in \Phi^{*} S$, so $S$ is approximately multiplicatively closed.
\end{proof}

Assume $R$ is a commutative approximately ring and the pair \boxed{(\Phi^{*}R,+),(\Phi^{*}R,\cdot)} are groupoids. Let $p$ be a non-approximately unit in $R$. The element $p$ is called an approximately principle prime, provided the non-zero ideal 
	\[(p)= \{ p\cdot k  \ \ : \   \ k \in R, \  p\cdot k \in R\}\]
is an approximately prime ideal in $R$.
	
\begin{theorem}
  Suppose $R$ is a commutative approximately ring, $p$ is a non-approximately unit in $R$ and $\Phi^{*}R$ is groupoid with binary operations \enquote{+} and \enquote{$\cdot$}. The non-zero set $(p)$ is an approximately ideal in $R$. 
\end{theorem}
\begin{proof}
  Let $x,y \in (p)$, then $x= p\cdot k_{1}$ and $y= p\cdot k_{2}$ where $k_{1}$ and $k_{2}$ are elements in $R$. Now $x-y=p\cdot (k_{1} -k_{2})$, $k_{1} - k_{2} \in \Phi^{*}R$, which implies that there exists $k \in R$ such that $\Phi(k) = \Phi(k_{1} - k_{2})$. Therefore, $p\cdot (k_{1} - k_{2}) \in \Phi^{*}(p)$, since $p\cdot k \in (p)$. Thus, $x-y \in \Phi^{*}(p)$.\newline
  Let  $r\in R$ and $x\in (p)$ so that  $x= p\cdot k_{1}$,  for some $k_1\in R$. Now $r\cdot x =r\cdot (p \cdot k_1) = (r\cdot p)\cdot k_{1}= (p\cdot r)\cdot k_{1}) =p \cdot (r\cdot k_1) \in \Phi^{*}(p)$.
\end{proof}
	
\begin{definition} {\bf (Approximately irreducible element)} \\
  An element $a \in R$ is called an approximately irreducible element if and only if $a$ is not an approximately unit and $a=b \cdot c$, where $bc \in \Phi^{*} R$. Then, $b$ is an approximately unit or $c$ is an approximately unit. \qquad \textcolor{blue}{$\blacksquare$} 
\end{definition}
		
\begin{lemma}
	Let $D$ be an approximately integral domain with identity. An approximately principle prime element $p$ is an approximately irreducible element.  
\end{lemma}
\begin{proof}
	Let $p$ be an approximately principle prime element in $R$, then $(p)$ is an approximately prime ideal. Assume $p$ is not an approximately irreducible element. Then if $p=a \cdot b$ where
	$a \cdot b \in \Phi^{*} D$ and neither $a$ nor $b$ is an approximately unit in $D$. Now $p=a \cdot b$ implies
	$a \cdot b \in \Phi^{*} (p)$. Since $(p)$ is an approximately prime ideal, then $a \in (p)$ or $b \in (p)$. Suppose $ a \in (p)$, hence $a=p\cdot k$ where $k \in D$. Now $p = a\cdot b = p\cdot (k\cdot b)$ we have $k \cdot b = 1_D$. Hence, $b$ is an approximately unit which contradicts our assumption. In a similar way if $b \in (p)$, then $a$ would be an approximately unit. Therefore, $p$ is an approximately irreducible element. 
\end{proof}

\begin{theorem}
  For a commutative approximately ring $R$. If every approximately ideal in $R$ (other than $R$) is approximately prime, then $R$ is an approximately integral domain.
\end{theorem}
\begin{proof}
	Let $a,b \in R$, such that $ab=0$ where $a\cdot b \in \Phi^{*}R$. Since every ideal other than $R$ is approximately prime, then the zero ideal $(0)$ must be an approximately prime. Now $a\cdot b \in \Phi^{*}(0)$ because $a\cdot b=0$, hence $a \in (0)$ or $b \in (0)$, which means either $a = 0$ or $b = 0$. Thus, $R$ is an approximately integral domain.
\end{proof}

\begin{theorem}
   Let $R$ be an approximately integral domain with identity. If every approximately ideal in $R$ (other than $R$) is approximately prime, then $R$ is an approximately field.    
\end{theorem}
\begin{proof}
  To prove that $R$ is an approximately field, it suffices to show that every non-zero element of $R$ has an approximately inverse.	
  Let $b$ be any non-zero element in $R$. Consider the approximately prime ideal $(b^2)$ where $b^2 \in (b^2)$ and $b^2 = b\cdot b$, implies $b \in (b^2)$. Hence, $b = b^{2} r$ where $r \in R$. Now we have $b\cdot 1_{R} - b^2\cdot r= 0$, then $ b\cdot (1_{R} - b\cdot r) = 0$, since $R$ is an approximately integral domain and $ b \neq 0$. So, $1_{R} - b\cdot r=0$, this implies $1_{R} = b\cdot r$ which means that $r$ is an approximately inverse of $b$ in $R$. Therefore, $R$ is an approximately field.
\end{proof}
	
\subsection{The product of two approximately ideals}

Suppose $A$ and $B$ are two approximately ideals in $R$. The product of two ideals $AB$ is defined by 
\[
AB= \left\{ \displaystyle \sum_{i=1}^{n} a_i\cdot b_i \ \ : \ \ a_i \in A, b_i \in B, \quad n\in\mathbb{Z}^+, \quad  a_i\cdot  b_i \in \Phi^{*}R, \sum_{i=1}^{n} a_ib_i \in R \right\}.
\] 
	
\begin{lemma}
  Let $R$ be an approximately commutative ring with identity. If $A$ and $B$ are approximately ideals of $R$, then $AB$ is an approximately ideal in $R$.     
\end{lemma}
\begin{proof}
   Let $x,y \in AB$, then $x = \displaystyle \sum_{i=1}^{n} a_i\cdot  b_i$ and $y = \displaystyle \sum_{j=1}^{m} \bar a_j \cdot  \bar b_j$. Assume $\bar a_{j} = -a_{j+n} \in A$ and $\bar b_{j} = b_{j+n} \in B$. Now $x - y = \displaystyle \sum_{i=1}^{n} a_i\cdot  b_i - \displaystyle \sum_{j=1}^{m} \bar a_j \cdot  \bar b_j = \displaystyle \sum_{i=1}^{n+m} a_i \cdot  b_i$,thus $x- y \in AB$, which implies $x-y \in \Phi^{*} AB$.
   Suppose $x \in AB$ and $r \in R$, then $x = \displaystyle \sum_{i=1}^{n} a_i\cdot  b_i$  we have to show that $rx \in \Phi^{*} AB$. Now $r\cdot  x = r\cdot \displaystyle \sum_{i=1}^{n} a_i\cdot  b_i$ $= \displaystyle \sum_{i=1}^{n} (r \cdot  a_i)\cdot  b_i$, since $ra_{i} \in \Phi^{*} A$ for all $i\in I$, we conclude that $r\cdot  x \in \Phi^{*} AB$.
\end{proof}

\begin{theorem}
   If $A$ is an approximately prime ideal in an approximately ring $R$ and $A=BC$, where $B$ and $C$ are approximately ideals in $R$, then either $B \subset A$ or $C \subset A$.
\end{theorem}
\begin{proof}
  Without loss of generality, suppose $B \not\subset A$, then there exists some $b_{1} \in A$ such that $b_1 \notin A$. We have $A=BC$, then in particular $b_{1}c_i \in \Phi^{*}A$ for all $c_i \in C$ since $A$ is an approximately prime ideal, then $b_1 \in A$ or $c_i \in A$, but $b_1 \notin A$. Therefore, $c_i \in A$ for all $c_i \in C$.Hence, $C \subset A$.
\end{proof}

\subsection{Approximately prime rings}
	We say that an approximately ring $R$ is prime if and only if $(0)$ is an approximately prime ideal in $R$.
	
\begin{theorem}
   An approximately ring $R$ is prime if and only if for all $a, b \in R, a\cdot r \cdot b = 0$ for all non-zero element $r\in R$, which implies $a = 0$ or $b = 0$. 
\end{theorem}
\begin{proof} 
It is easy to show that for any $a,b \in R$, if $a\cdot r\cdot b=0$ for all non-zero element $r\in R$, implies $a= 0$ or $b=0$. Thus, $R$ is an approximately prime ring. Now we want to prove the converse.
Assume $R$ is an approximately prime ring and $a,b \in R$, $a\cdot r\cdot b = 0$ for all non-zero element $r \in R$, then $arb \in (0)$, since $R$ is an approximately prime ring, we have $(0)$ is an approximately prime ideal, i.e., if $arb \in \Phi^{*}(0)$. Hence, we can conclude that $a = 0$ or $b = 0$.
\end{proof}
	
\begin{theorem}
  $I$ is an approximately prime ideal of an approximately ring $R$, provided $R/I$ is an approximately prime ring.
\end{theorem}
\begin{proof}
  Suppose $I$ is an approximately prime ideal of an approximately ring $R$. Let $\bar a, \bar b \in R/I$,
then $\bar a = a + I$ where $a \in R$ and $\bar b = b+I$ where $b \in R$. Assume $\bar a (R/I) \bar b = (0)$ then $\bar a \cdot  \bar r \cdot  \bar b = \bar 0$, for any non-zero $\bar r \in R/I$. Now we have $(a + I)\cdot (r + I)\cdot (b + I) = I $ where $a,r $ and $b \in R$. Thus, $a\cdot r\cdot b + I = 0 + I,$ so $ a\cdot r\cdot b \in \Phi^{*}(I)$, then $a \in I$ or $b\in I$ since $I$ is an approximately prime ideal, $a \in  I$ or $b \in I$, implies $\bar a = \bar 0$ or $\bar b = \bar 0$. Therefore, $R/I$ is an approximately prime ring.
\end{proof}

\subsection{Approximately direct product}
    Suppose $R_1$ and $R_2$ are approximately rings.  The direct product of $R_1$ and $R_2$ is defined by
	\[R_1 \times R_2 = \{(r_1,r_2 ) \ \ : \ \  r_1 \in R_1  \quad\mbox {and} \quad r_2\in R_2\}.\]
	Here, for $(r_1,r_2)$ and $(\bar r_1, \bar r_2) \in R_1\times R_2$, the addition operation is defined by 
	\[(r_1,r_2 )+(\bar r_1 ,\bar r_2 )= (r_1+ \bar r_1 ,r_2+\bar r_2 )\] 
	where $r_1+ \bar r_1 \in \Phi^{*} R_1$ and $r_2+ \bar r_2 \in \Phi^{*} R_2$.
	Further, the multiplication operation is defined by 
	\[
	(r_1,r_2 )*(\bar r_1 ,\bar r_2 )= (r_1 * \bar r_1 ,r_2 *\bar r_2 ),
	\]
	where $r_1 * \bar r_1 \in \Phi^{*} R_1$ and $r_2 * \bar r_2 \in \Phi^{*} R_2$.\\

\begin{theorem}
  If $R_1 $ and $R_2$ are approximately rings, $\Phi^{*}R_1$, $\Phi^{*} R_2$ are groupoids with the binary operations \enquote{+}  and \enquote{*}. If  $\Phi^{*} (R_1 \times R_2) = \Phi^{*} R_1 \times \Phi^{*}R_2$, then $R_1 \times R_2$ is an approximately ring. 
\end{theorem}
\begin{proof}$\mbox{}$\\
	\begin{description}
	\item[{\bf ($AR_1$)}] Let $\Phi^{*} (R_1 \times R_2) = \Phi^{*} R_1 \times \Phi^{*}R_2$. Since $R_1 $ and $R_2$ are approximately rings, then $(R_1 \times R_2,+)$ is an abelian approximately group.
	\item[{\bf ($AR_2$)}] Since $(R_1,*)$ and $(R_2,*)$ are approximately semi-group, we conclude that $(R_1 \times R_2,*)$ is an approximately semi-group. 
    \item[{\bf ($AR_3$)}] For $(a,b)$, $(c,d)$ and $(m,n)$ are elements in $R_1 \times R_2$, then we can say that $(m,n)*((a,b)+(c,d))=(m,n)*(a,b) + (m,n)*(c,d)$ holds in $\Phi^{*} (R_1 \times R_2)$. 
	In similar way, we can conclude that the right distributive law is also satisfied. 
\end{description}
\end{proof}

\begin{theorem}
The approximately direct product of two approximately rings can never be an approximately prime ring.
\end{theorem}
\begin{proof}
Let $R=R_1 \times R_2$ be an approximately direct product of two approximately rings. In that case $R$ is not an approximately prime ring. For $a = (1,0)$ and $b = (0,1)$ we have  $aRb = (0)$ for all non-zero $(r_1, r_2) \in R$. However, $a \neq 0$ and $b \neq 0$, so $R$ can never be approximately prime ring.
\end{proof}

\begin{corollary}
	An approximately direct product of two approximately rings can never be an approximately integral domain.  
\end{corollary}
		
\begin{proposition}
   Let $R = R_1 \times R_2$ be an approximately direct product of two approximately rings. If $P = R_1 \times \{0\}$ is an approximately prime ideal, then $R_2$ is an approximately integral domain.  
\end{proposition}
\begin{proof}
  Suppose $R_2$ is not an approximately integral domain. Then there exist non-zero elements $r_2, s_2 \in R_2$ such that $r_2s_2 = 0$. Let $(0, r_2), (0, s_2) \in R$ such that $(0, r_2)(0, s_2) \in \Phi^{*}(P)$. Now we have $(0, r_2) \in P$ or $(0, s_2) \in P$, since $(P)$ is an approximately prime ideal. Thus, $r_2 = 0$ or $s_2 = 0$, but that contradicts with the fact that $r_2$ and $s_2$ are non-zero elements. Therefore, $R_2$ is an approximately integral domain.  
\end{proof}

\end{document}